\documentclass[11pt,a4paper]{article}
\usepackage[utf8]{inputenc}
\usepackage[english]{babel}
\usepackage{empheq}
\usepackage{amsmath}
\usepackage{amsfonts}
\usepackage{amssymb}
\usepackage{amsthm}
\usepackage{graphicx}
\usepackage[framemethod=TikZ]{mdframed}
\usepackage[margin=1.2in]{geometry}
\usepackage{soul}
\usepackage{hyperref}
\usepackage{authblk}
\usepackage{mathtools}
\usepackage{esint}
\usepackage{ulem}

\usepackage{xcolor}

\newcommand{\ep}{\epsilon}

\newtheorem{theorem}{Theorem}[section]
\newtheorem{lemma}[theorem]{Lemma}
\newtheorem{claim}[theorem]{Claim}
\newtheorem{proposition}[theorem]{Proposition}
\newtheorem{corollary}[theorem]{Corollary}

\theoremstyle{definition}
\newtheorem{definition}[theorem]{Definition}

\theoremstyle{remark}
\newtheorem{remark}[theorem]{Remark}

\numberwithin{equation}{section}

\begin{document}
\selectlanguage{english}

\title{Ideal Free Dispersal under General Spatial Heterogeneity and Time Periodicity}
\author{Robert Stephen Cantell and Chris Cosner} 
\affil{Department of Mathematics, University of Miami\\
\texttt{rsc@math.miami.edu}, \texttt{gcc@math.miami.edu}}
\author{King-Yeung Lam}
\affil{Department of Mathematics, Ohio State University, Columbus, OH, United States\\
\texttt{lam.184@math.ohio-state.edu}}

\date{}
\maketitle
\thispagestyle{empty}

\affil[$\star$]{}


\medskip
%%    \subjclass is required.
%\subjclass[2010]{35K57, 92D15,  92D25}
\noindent Mathematics Subject Classification (2010) \quad 35K57$\,\dot\,$35Q92$\,\cdot\,$92D15$\,\cdot\,$92D25
%
%\date{\today}
%
%%\dedicatory{}
%
%%    Abstract is required.
\begin{abstract}
{A population is said to have an ideal free distribution in a spatially heterogeneous but temporally constant environment if each of its members have chosen a fixed spatial location in a way that optimizes its individual fitness, allowing for the effects of crowding. In this paper, we extend the idea of individual fitness associated with a specific location in space to account for the full path that an individual organism takes in space and time over a periodic cycle, and extend the mathematical formulation of an ideal free distribution to general time periodic environments.  We find that, as in many other cases,  populations using dispersal strategies that can produce a generalized ideal free distribution have a competitive advantage relative to populations using strategies that do not produce an ideal free distribution. A sharp criterion on the environmental functions is found to be necessary and sufficient for such ideal free distribution to be feasible. In the case the criterion is met, we showed that there exist dispersal strategies that can be identified as producing a time-periodic version of an ideal free distribution, and such strategies are evolutionarily steady and are neighborhood invaders from the viewpoint of adaptive dynamics. Our results extend previous works in which the environments are either temporally constant, or temporally periodic but the total carrying capacity is temporally constant.} 
\end{abstract}
%
%\maketitle
%      Mathematically, our results extend those of \cite{Averill2012, Cantrell2010, Cantrell2018} to more general forms of logistic-type models and to time-periodic environments where the total amount of resources available in the environment can vary in time.  Specifically, we extend the idea of fitness associated with a specific location in space to account for the full path that an organism takes in space and time over a periodic cycle, and extend the mathematical formulation of an ideal free distribution to general time periodic environments.  We find that as  in many  other cases,  populations using dispersal strategies that can produce  a (generalized) ideal free distribution have a competitive advantage relative to populations using strategies that do not produce an ideal free distribution.

\section{Introduction}\label{sec:1}

The ideal free distribution is by now a well established concept in ecological theory, with profound ramifications in the understanding of evolution of dispersal \cite{Averill2012,Cantrell2010,Cantrell2012JMB,Cantrell2012CAMQ,Cantrell2017,Cosner2014, Korobenko2009,Korobenko2012,Korobenko2014}. The ideal free distribution was initially formulated as a verbal description of the way organisms located themselves in nature \cite{Fretwell1972,Fretwell1969} motivated by observing territorial patterns of birds. It asserts that if the members of a species have complete knowledge of the environment (ideal) and may locate themselves without cost (free), then they will do so in a manner that maximizes fitness, here thought of as local per capita reproductive success. Fitness is assumed to be limited by the presence of conspecifics at the same spatial location. In this framework, an ideal free distribution is achieved by the population when no individual can improve its fitness by moving in a different way.

Suppose the environment is spatially heterogeneous but temporally constant, and that dispersal and population dynamics are coupled additively. In this case one can argue formally that the ideal free distribution is equivalent to fitness being equilibrated at zero. To see that, first we make the reasonable assumption that, like the environment, the ideal free distribution of population is also temporally constant. Then the species in question should continue to increase in abundance so long as fitness remains positive. Therefore, the local population growth should be zero, and the ideal free distribution corresponds to a spatially varying equilibrium of the system in the absence of dispersal. In the particular situation of a mathematical model with logistic growth in a habitat with a favorable resource distribution, under appropriate scaling, such an equilibrium aligns perfectly to the carrying capacity. Indeed, this has been rigorously proven in a number of modeling settings, including reaction-diffusion-advection \cite{Averill2012,Cantrell2010,Korobenko2009,Korobenko2012,Korobenko2014}, discrete diffusion \cite{Cantrell2012JMB}, integro-differential \cite{Cantrell2012CAMQ,CDM2012}, and integro-difference models \cite{Cantrell2018pre}. In these papers,  the ideal free dispersal, manifesting itself as the perfect alignment with carrying capacity, confers a distinct evolutionary advantage to a residential species adopting such strategy. This advantage is expressed through the parlance of adaptive dynamics \cite{odo,Geritz2008,Geritz1998} and focuses on the pairwise invasibility of competing species. We say that a dispersal strategy is evolutionarily
stable (ESS), also known as evolutionarily steady, relative to some classes of strategies if a species adopting this strategy cannot be invaded by an ecologically identical competitor adopting any other strategy from this class \cite{MP}. Precisely, when a rare population of competitors is being introduced into an environment, in which the resident playing the ESS is at {equilibrium}, the population of competitors decays in time. On the other hand, a strategy is a neighborhood invader strategy (NIS) if it can invade any nearby strategy. Strategies which are both ESS and NIS have a clear evolutionary advantage. The results in \cite{Averill2012,Cantrell2010,Cantrell2012JMB,Cantrell2012CAMQ,Cantrell2017,Cantrell2018pre,Cosner2014, Korobenko2009,Korobenko2012,Korobenko2014} show that ideal free dispersal is both an ESS and a NIS robustly across a range of mathematical modeling frameworks in spatially heterogeneous but temporally constant environments.

%{\color{blue}The progress on temporally fluctuating environments has been limited, until the recent work \cite{Cantrell2018}. Under the assumption that the carrying capacity is positive everywhere and the total carrying capacity is independent of time, it was shown that IFD, in the form of perfect alignment with carrying capacity, can actually be realized. i.e. If the carrying capacity $K(x,t)$ satisfies 
%\begin{equation}\label{eq:perfect}
%K(x,t) >0 \quad \text{ for all }x,t, \quad \text{ and }\quad \int K(x,t)\,dx = \text{const.}
%\end{equation}
%then there exists dispersal strategy such that the corresponding steady population density matches $K(x,t)$ perfectly. Note that \eqref{eq:perfect} is indeed a necessary condition for the population to be able to perfectly match the carrying capacity, in the sense that zero net birth or death rates are being enforced. Furthermore, it was shown that such IFD strategies are ESS and NIS.}

For the existence of ESS in families of dispersal strategies that do not generate IFD, we refer to \cite{Lam2014a,Lam2014b} and the book chapter \cite{LamLouChapter} for results in the adaptive dynamics framework based on two-species interactions. See also \cite{Hao2017,Lam2017b,Lam2017a} for related results in a framework in which the population is structured by a dispersal trait and space. In the latter framework, the interaction between an infinite number of species is being investigated. 

{The case of environments that are heterogeneous in both space and time differs from the temporally constant case.  In the temporally constant case, a logistically growing population that is initially distributed with a positive density everywhere will grow to exactly match the local carrying capacity everywhere.  Thus, if the initial  conditions are right, a population can achieve an ideal free distribution with no dispersal at all.  It is well known that in the case of simple diffusion in continuous environments, or discrete diffusion in patchy environments according to a fixed dispersal matrix,  there is selection for slower diffusion, and indeed the strategy  of not moving at all is convergence stable (CSS) in the sense that a population with a smaller diffusion rate can invade an ecologically similar population with a faster diffusion rate \cite{Hastings1983, DHMP}.  This fact is related to the more general observation that dispersal resulting in mixing across space reduces growth rates in a wide range of modeling contexts in temporally constant but spatially varying environments,  which is known as the reduction phenomenon \cite{Altenberg}.   In environments that are heterogeneous in both time and space,  there may sometimes be selection for faster diffusion, or there may be stable polymorphisms where competitors with different dispersal rates coexist;  see \cite{Hutson}.   Furthermore, if the local carrying capacity or resource distribution of an environment varies in both space and time, a population cannot match it without some type of dispersal. The motivation for trying to understand the ideal free distribution in time periodic environments is strong.  Many environments are seasonal, and in fact  one of Fretwell's original  works on the ideal free distribution, \cite{Fretwell1972}, is entitled Populations in a Seasonal Environment.  In the special case of a logistic model in an environment where the total carrying capacity of the environment (or alternatively the sum of available resources) across space is constant in time, but whose spatial distribution may vary periodically in time, a dispersal strategy that allows a population to match the  resource distribution exactly is derived and shown to be an ESS and NIS in \cite{Cantrell2018}.  This case differs from the temporally constant case in that to match the resource distribution  in space and time requires the use of nonlocal information about the environment, which is not the case in  static environments, where dispersal strategies based on purely local cues can produce an ideal free distribution.}
 
{There is a large literature on periodic-parabolic models in ecology and population dynamics.  For general background, see \cite{Hess1991}.  More recent results can be found in \cite{ALG1,ALG2,Liu} and the references therein.  An interesting recent application to protection zones is given in \cite{Lopez2019}. There also  has been some work on nonspatial models for populations in periodically varying environments where the populations are structured by a trait that is subject to random mutation. Those lead to periodic-parabolic  reaction-diffusion equations where the diffusion is in trait space rather than physical space; see for example   \cite{CN,Iglesias2018}.}

{There are a number of issues that must be addressed in trying to interpret the concept of an ideal free distribution and understand the evolution of dispersal in environments that vary in both time and space.  It is not immediately clear how individual fitness or an ideal free distribution should be defined in that setting.  In static environments we can view the local population growth rate at a given location as a proxy for the fitness of individuals at that location. However, that definition is inadequate in time varying environments where the fitness of an individual depends not only on its location but also on how it is moving through space and time. Similarly, an ideal free distribution in a static environment can be defined in terms of  habitat selection that optimizes fitness, but in a varying environment phenology (that is, the timing of life history events including migration)  also becomes important. }

{In this article, we continue the investigation of evolution of dispersal strategies in general spatially heterogeneous and time-periodic environments. As is shown by \cite{Cantrell2018}, the analogy between IFD and perfect alignment with carrying capacity can hold only under the limitation that the total carrying capacity is constant in time. In Section \ref{sec:2}, we establish a notion of pathwise fitness, which leads to a broader notion of IFD in the spatially heterogeneous and temporally periodic context. In particular, the generalized version of IFD reduces to and includes the results of \cite{Cantrell2018} when the condition \eqref{eq:perfect} holds. 
In Section \ref{sec:4} we propose a necessary and sufficient condition for the existence of dispersal strategies that enable the species to achieve IFD, and construct a concrete class of such dispersal strategies. 
In Section \ref{sec:3}, we will show that the proposed notion of IFD 
confers the same evolutionary advantage as in the temporally constant case, {in the sense that once} a strategy enables the species to achieve IFD, then it is both ESS and NIS. Morevoer, we will also show that it is necessary in the sense that a time-periodic ecological attractor that is not IFD can always be invaded by an exotic species equipped with a suitably chosen dispersal strategy. In Section \ref{sec:5}, we close with some discussion. }

%
%In this article, we will
%\begin{itemize}
%    \item Define a notion of individual fitness in a spatially heterogeneous, time-periodic environment;
%    \item Define what it means to be an ideal free distribution in such environments;
%    \item Establish that when a strategy enables a species to achieve the ideal free distribution, then that strategy is both ESS and NIS;
%    \item Construct dispersal strategies which enable the species to achieve ideal free distribution.
%    
%\end{itemize}

\section{Pathwise fitness and IFD in spatially heterogeneous and temporally periodic environments}\label{sec:2}

\subsection{The single species model}
We consider the following class of reaction-diffusion-advection models in the spatially heterogeneous and temporally periodic setting, which models the dynamics of the density $\theta(x,t)$ of a single species:
\begin{equation}\label{eq:theta}
\left\{
\begin{array}{ll}
\frac{\partial \theta}{\partial t} = \nabla \cdot [\mu(x,t) \nabla \theta - \theta \vec{P}(x,t)]  + r(x,t)\theta\left( 1 - \frac{\theta}{K(x,t)}\right) &\text{ in }\Omega \times (0,\infty),\\
n \cdot [\mu(x,t) \nabla \theta - \theta \vec{P}(x,t)]  = 0 &\text{ on }\partial\Omega \times (0,\infty),\\
\theta(x,0) = \theta_0(x) &\text{ in }\Omega.
\end{array}
\right.
\end{equation}
where $\Omega$ is a bounded domain in $\mathbb{R}^n$ with smooth boundary $\partial\Omega$ and unit outer normal vector $n=n(x)$; $\mu(x,t)$ is the diffusion rate; $\vec{P}(x,t)$ the vector field describing the directed movement; $r(x,t)$ is the local intrinsic growth rate; $K(x,t)$ is the carrying capacity, and $n$ denotes the outer unit normal vector on $\partial\Omega$.  Furthermore, the environment $(r,K)$ and dispersal strategies $(\mu,\vec{P})$ are assumed to be $T$-periodic in $t$. We assume that there is no net movement into or out of the domain, as is reflected by an application of divergence theorem:
\begin{equation}
\int_\Omega \nabla\cdot[\mu(x,t) \nabla \theta - \theta \vec{P}(x,t)]\,dx= 0.
\end{equation}
It is well known that, for given positive, {$T$-periodic coefficients $\mu,r,K$ and $T$-periodic vector field $\vec{P}$}, the equation \eqref{eq:theta} has a unique positive periodic solution $\theta^*(x,t)$, which is globally asymptotically stable among all nonnegative, nontrivial solutions \cite{Hess1991}. %{\color{blue}Without loss of generality, we will assume $|\Omega|=1$.}

\subsection{A notion of fitness}

When the environment is temporally constant (i.e. $r=r(x)$ and $K=K(x)$), the stable population will be at equilibrium $\theta^*(x)$, so that the reproductive fitness can be understood as a function depending on location only:
$$
F(x) = r(x) \left( 1 - \frac{\theta^*(x)}{K(x)}\right).
$$
In such a case the IFD is realized when fitness is equilibrated, so that the IFD equilbrium corresponds to the perfect alignment with carrying capacity:
$$
\theta^*(x) \, \equiv \, K(x).
$$
In contrast, when the environment is time-periodic, we expect the population to stabilize at a time-periodic solution $\theta^*(x,t)$, so that the reproductive fitness function depends nontrivially on both space and time:
$$
F(x,t) = r(x,t) \left( 1- \frac{\theta^*(x,t)}{K(x,t)}\right).
$$
In case $r(x,t)=K(x,t)$, the first two authors \cite{Cantrell2018} show that, under the additional assumption that the carrying capacity is everywhere positive and that the total level of resources in the environment remains constant in time, precisely,  %\eqref{eq:perfect} holds,
\begin{equation}\label{eq:perfect}
K(x,t) >0 \quad \text{ for all }x,t, \quad \text{ and }\quad \int K(x,t)\,dx = \text{const.},
\end{equation}
there exists a class of dispersal strategies (i.e. choices of $\mu$ and $\vec{P}$) under which the ideal free distribution can be achieved in the form of perfect alignment with carrying capacity, i.e. $\theta^*(x,t) = K(x,t)$. The authors went on to show that such dispersal strategies are both ESS and NIS. However, as observed in \cite{Cantrell2018}, perfect alignment with carrying capacity is impossible if \eqref{eq:perfect} is false. Indeed, consider the simplest situation of a single patch with $r(t)=K(t)$ being nonconstant and $T$-periodic in $t$. Being in a single patch means there is no net movement, so that the population will stabilize at the unique positive periodic solution $\theta^*(t)$ satisfying
$$
\frac{d}{dt}\theta^*(t) = (K(t) - \theta^*(t))\theta^*(t)
$$
which clearly cannot satisfy $\theta^*(t) = K(t)$ for all $t$.  The same holds for environments which are spatially homogeneous and temporally periodic.

Here, we propose the pathwise fitness of an individual traveling a path $\gamma(t)$ during the time period $[0,T]$ to be
$$
F_{\rm path}(\gamma) = \int_0^T \left[r(x,t) \left( 1 - \frac{\theta^*(x,t)}{K(x,t)}\right) \right]_{x=\gamma(t)}\,dt.
$$
Suppose the population achieves ideal free distribution, then we expect that no individual can gain proliferative advantage by dispersing itself differently within the time period $[0,T]$, so that the pathwise fitness $F_{\rm path}(\gamma)$ will be equilibrated among all possible paths $\gamma$. This motivates the following definition.
\begin{definition}\label{def:IFD}
Suppose the population stabilizes at a positive periodic solution $\theta^*(x,t)$, we say that $\theta^*(x,t)$ is an ideal free distribution (or IFD in short) if
$$
F(x,t) =  r(x,t) \left( 1- \frac{\theta^*(x,t)}{K(x,t)}\right) \quad \text{ is independent of }x.
$$
\end{definition}
Observe that
$$
\inf_\gamma F_{\rm path}(\gamma) = \int_0^T \inf_{x \in \Omega} F(x,t)\,dt \quad \text{ and }\quad \sup_\gamma F_{\rm path}(\gamma) = \int_0^T \sup_{x \in \Omega} F(x,t)\,dt,
$$
where the infimum and supremum are taken over all continuous paths $\gamma: [0,T] \to \Omega$. (We remark that the set of test functions can be further reduced to the class of all smooth periodic paths, by a density argument.) Hence,
 a distribution $\theta^*(x,t)$ is IFD if and only if the corresponding pathwise fitness $F_{\rm path}(\gamma)$ is independent of paths $\gamma$. Observe that total alignment with resource (i.e. $\theta^*(x,t) = K(x,t)$ for all $x$ and $t$) is a sufficient, but not necessary, condition for IFD in the sense of Definition \ref{def:IFD}.

\section{The existence of IFD strategy}\label{sec:4}

%Consider \eqref{eq:theta}, which describes a single species equipped with the dispersal strategy with diffusion rate $\mu(x,t)$ and directed movement $\vec{P}(x,t)$. 
%It is well-known that \eqref{eq:theta} admits a unique positive $T$-periodic solution $\theta^*(x,t)$ which is globally asymptotically stable with respect to all nonnegative, nontrivial solutions.

The goal of this section is to give a sufficient and necessary condition, in terms of the environmental functions $r(x,t)$ and $K(x,t)$, so that there exists a class of dispersal strategies $(\mu,\vec{P})$ whose corresponding positive $T$-periodic solution of \eqref{eq:theta} is an IFD. 

{We outline our main ideas as follows. First, we define a periodic function $M(t)$ by solving a Bernoulli-type equation (see Lemma \ref{lem:M}). Second, we consider the transformation $\tilde\theta(x,t)=\theta(x,t)/M(t)$, which 
transforms the single species problem \eqref{eq:theta} into the following:
\begin{equation}\label{eq:thetap}
\left\{
\begin{array}{ll}
\frac{\partial \tilde\theta}{\partial t} = \nabla \cdot [\mu(x,t) \nabla \tilde\theta - \tilde\theta \vec{P}(x,t)]  + \tilde r(x,t)\tilde\theta\left( 1 - \frac{\tilde\theta}{\tilde{K}(x,t)}\right) &\text{ in }\Omega \times (0,\infty),\\
n \cdot [\mu(x,t) \nabla \tilde\theta - \tilde\theta \vec{P}(x,t)]  = 0 &\text{ on }\partial\Omega \times (0,\infty),\\
\tilde\theta(x,0) = \tilde\theta_0(x) &\text{ in }\Omega,
\end{array}
\right.
\end{equation}
where
\begin{equation}\label{eq:rp}
\tilde{r}(x,t) = r(x,t) - \frac{M'(t)}{M(t)}, % \quad \text{ and }\quad \tilde{K}(x,t)= \frac{K(x,t)M(t)}{r(x,t)}\hat{r}(x,t)= \frac{K(x,t)}{M(t)} - \frac{K(x,t)}{r(x,t)}\cdot \frac{M'(t)}{M(t)^2}.
\end{equation}
and
\begin{equation}\label{eq:m}
\tilde{K}(x,t)= \frac{K(x,t)}{M(t)r(x,t)} \left( r(x,t) - \frac{M'(t)}{M(t)}\right) = \frac{K(x,t)}{M(t)} - \frac{K(x,t)}{r(x,t)}\cdot \frac{M'(t)}{M(t)^2}.
\end{equation}
Then we have an equivalent system where the new carrying capacity $\tilde{K}(x,t)$ satisfies, by our choice of $M(t)$ in Lemma \ref{lem:M}, 
\begin{equation}\label{eq:m2m}
\int \tilde{K}(x,t)\,dx = 1 \quad \text{ for all }t.
\end{equation}
Provided that the positivity conditions $\tilde{r},\tilde{K}>0$ (see \eqref{eq:R}) hold, we can apply the results in \cite{Cantrell2018} to yield a IFD of \eqref{eq:thetap} that perfectly matches the new carrying capacity $\tilde{K}(x,t)$. This gives a mathematically natural way of extending the notion of IFD to general time-periodic environments. The condition \eqref{eq:R}, which corresponds to the postivity of carrying capacity in the transformed problem \eqref{eq:thetap}, can be interpreted as a criterion that prevents the appearance of generalized sinks in the original problem \eqref{eq:theta}.

In Subsection \ref{subsect:2.1}, we state the condition \eqref{eq:R}, under which we can define a IFD $\theta^*(x,t) = M(t)\tilde{K}(x,t)$, where $M(t)$ and $\tilde{K}(x,t)$ are defined in terms of the environmental data $r(x,t), K(x,t)$ by solving a Bernoulli-type equation. In Subsection \ref{subsect:2.2}, we will show that the condition \eqref{eq:R} is necessary, and that the IFD $\theta^*(x,t)=M(t)\tilde{K}(x,t)$ is the only possible form of IFD for the class of population dynamics as described by the reaction-diffusion-advection equation \eqref{eq:theta}. In Subsection \ref{subsect:2.3}, we construct a class of dispersal strategies which generates IFD.}

\subsection{The Bernoulli equation}\label{subsect:2.1}
%We first prepare with a couple of definitions.
%Given $r(x,t), K(x,t)$ we define $a(t),b(t),M(t)$ and $\tilde{K}(x,t)$ as follows.
%\begin{equation}\label{eq:ab}
%a(t):= \frac{\overline{K}(t)}{\overline{(K/r)}(t)}\quad \text{ and }\quad b(t) = \frac{1}{\overline{(K/r)}(t)},
%\end{equation}
%where hereafter, we denote by $\bar{f}(t)$ the spatial average of $f(x,t)$, as given by $\frac{1}{|\Omega|}\int_\Omega f(x,t)\,dx$. 
Given the environmental functions $r(x,t), K(x,t)$ we define $M(t)$ as follows.
\begin{lemma}\label{lem:M}
There exists a unique, positive, $T$-periodic function $M(t)$ such that
\begin{equation}\label{eq:MMM}
\frac{\overline{K}(t)}{M(t)} -  \overline{(K/r)}(t)\frac{M'(t)}{M(t)^2} = 1.
\end{equation}
where hereafter, we denote by $\bar{f}(t)$ the spatial average of $f(x,t)$, as given by $\frac{1}{|\Omega|}\int_\Omega f(x,t)\,dx$. %In particular, it holds that
%\begin{equation}\label{eq:M2M}
%\frac{\overline{K}(t)}{M(t)} + \frac{M'(t)}{M(t)} \overline{K/r}(t) = 1.
%\end{equation}
\end{lemma}
\begin{proof}
It remains to solve
\begin{equation*}%\label{eq:MMM}
\frac{M'(t)}{M(t)} + b(t) M(t) = a(t).
\end{equation*}
where
\begin{equation}\label{eq:ab}
a(t):= \frac{\overline{K}(t)}{\overline{(K/r)}(t)}\quad \text{ and }\quad b(t) = \frac{1}{\overline{(K/r)}(t)}.
\end{equation}
First we show uniqueness. Suppose $M(t)$ is a positive $T$-periodic solution of \eqref{eq:MMM}, then $w(t) = 1/M(t)$ satisfies
$$
-w' + b(t) = a(t)w
$$ 
so that
$$
\left[ \exp\left( \int_0^t a(s)\,ds\right) w(t)\right]' = b(t) \exp\left( \int_0^t a(s)\,ds\right).
$$
Integrating the above, we obtain
$$
w(t) = \exp\left( -\int_0^t a(s)\,ds\right) w(0) + \int_0^t b(s) \exp\left( -\int_s^t a(\tau)\,d\tau\right)\,ds.
$$
By invoking the periodicity $w(0) = w(T)$, we further determine $w(0)$ uniquely. That is,
\begin{align}
\frac{1}{M(t)} = w(t) &=  \exp\left( -\int_0^t a(s)\,ds\right) \frac{\int_0^T b(s) \exp\left( -\int_s^T a(\tau)\,d\tau\right)\,ds}{1 - \exp\left( -\int_0^T a(s)\,ds\right)} \notag \\
&\quad + \int_0^t b(s) \exp\left( -\int_s^t a(\tau)\,d\tau\right)\,ds.\label{eq:MT}
\end{align}
This proves uniqueness. Conversly, \eqref{eq:MT} also defines a solution to \eqref{eq:MMM}, so that existence follows immediately.
\end{proof}
%\begin{align}
%\frac{1}{M(t)} &= \exp\left(- \int_0^t a(s)\,ds\right) \frac{\int_0^T b(s)\exp\left(-\int_s^T a(\tau)\,d\tau\right)ds}{1- \exp\left(-\int_0^T a(s)\,ds\right)}  \notag \\
% &\qquad + \int_0^t b(s) \exp\left(-\int_s^t a(\tau)\,d\tau\right)\,ds.\label{eq:MT}
%\end{align}
Next, we define $\tilde{K}(x,t)$ in terms of $r(x,t), K(x,t)$ and $M(t)$, by \eqref{eq:m},
%\begin{align}
%\tilde{K}(x,t) =\frac{ K(x,t)}{M(t)} - \frac{K(x,t)}{r(x,t)} \cdot\frac{M'(t)}{M(t)^2} 
%\label{eq:m}
%\end{align}
and define
\begin{equation}\label{eq:define_ifd}
\theta^*(x,t):= M(t)\tilde{K}(x,t).
\end{equation}
We show that $\theta^*$  defines an ideal free distribution.
\begin{lemma}\label{lem:IFD}
Assume 
\begin{equation}\label{eq:R}
%\frac{K(x,t) - \overline{K}(t)}{M(t)} >    \frac{\overline{K}(t)}{M(t) } \left[ \frac{(K/r)(x,t)}{\overline{(K/r)}(t)} - 1\right] - \frac{(K/r)(x,t)}{\overline{(K/r)}(t)}.
K(x,t) > \left( K/r\right)(x,t) \frac{M'(t)}{M(t)}
\end{equation}
Then the distribution $\theta^*$ is an IFD. Precisely, the fitness function satisfies
\begin{equation}\label{eq:misifd}
r(x,t)\left[ 1 - \frac{\theta^*(x,t)}{K(x,t)}\right]=r(x,t)\left[ 1 - \frac{M(t)\tilde{K}(x,t)}{K(x,t)}\right] = \frac{M'(t)}{M(t)}.
\end{equation}
\end{lemma}
\begin{proof}
By \eqref{eq:R}, the function $\tilde{K}(x,t)$ defined in \eqref{eq:m} is positive, so that $\theta^*(x,t) = M(t)\tilde{K}(x,t)$ is also positive. Finally, \eqref{eq:misifd} follows by rewriting \eqref{eq:m}.
\end{proof}
\begin{remark} 
The condition \eqref{eq:R} can also be written as $r(x,t) > M'(t)/M(t)$. We prefer the above formulation, in view of the fact that $K$ and $K/r$ have already appeared in the definition of $M(t)$. Also, by combining \eqref{eq:m} and \eqref{eq:MMM}, we have
\begin{equation}\label{eq:mm}
\tilde{K}(x,t) = 
%\frac{1}{M(t)} \left[ K(x,t) - \frac{(K/r)(x,t)}{\overline{(K/r)}(t)}(\overline{K}(t) - M(t)) \right]
\frac{K(x,t) - \overline{K}(t)}{M(t)} + \frac{(K/r)(x,t)}{\overline{(K/r)}(t)}-  \frac{\overline{K}(t)}{M(t) } \left[ \frac{(K/r)(x,t)}{\overline{(K/r)}(t)} - 1\right].
\end{equation}
This gives yet another equivalent formulation of \eqref{eq:R}, which is
\begin{equation}\label{eq:R3}
\frac{K(x,t) - \overline{K}(t)}{M(t)} >    \frac{\overline{K}(t)}{M(t) } \left[ \frac{(K/r)(x,t)}{\overline{(K/r)}(t)} - 1\right] - \frac{(K/r)(x,t)}{\overline{(K/r)}(t)}.
\end{equation}
\end{remark}

\begin{remark}\label{rmk:2.4}
\begin{itemize}
\item[{\rm(a)}] If $r\equiv 1$, then \eqref{eq:R3} and thus \eqref{eq:R} holds for any $K(x,t)>0$. The corresponding IFD is given by
 $$
\theta^*(x,t)= \frac{M(t)}{\overline{K}(t)} K(x,t).
 $$
 Note that  $\theta^*(x,t)$ is proportional in space to $K(x,t)$, for each $t$.
  
\item[{\rm(b)}] If $K \equiv 1$, then $a(t) = b(t) = [\overline{(1/r)}(t)]^{-1}$ and $M(t) \equiv 1$ and the IFD is homogeneous 1,
%\sout{, since by \eqref{eq:mm}}
%$$
%\sout{[1-1] + \frac{1/r}{\overline{(1/r)}}\cdot 1 - \frac{1}{\overline{(1/r)}}[1/r - {\overline{(1/r)}}]=1}
%$$
%so $m\equiv 1$ and \eqref{eq:R} holds for sure. Here 
and the IFD strategy can simply be taken to be the homogeneous diffusion operator.

\item[{\rm(c)}]{If we assume $\overline{(K/r)}(t) \equiv 1$ and $\overline{K}(t) = \text{const.}$, then $M(t) = \overline{K}(t)$, so that \eqref{eq:R3} and \eqref{eq:R} 
is reduced to the requirements $r(x,t),K(x,t)>0$. This case includes the results in \cite{Cantrell2018}.}

\item[{\rm(d)}] {Consider the special case $r(x,t)\equiv K(x,t)$, which is the simplified logistic model considered in \cite{Cantrell2018}.  Fix an arbitrary $\rho(t)$ which is positive, periodic, and non-constant, and define
 $M(t)$ to be the unique periodic solution to
$$
\frac{M'(t)}{M(t)} + M(t) = \rho(t).
$$
Then it is easy to see that $M'(t)/M(t)$ changes sign in $t$. Hence, 
the condition \eqref{eq:R} fails if we choose $K(x,t)=r(x,t)$ such that 
$$
\bar{r}(t) = \rho(t),\quad \text{ and }\quad  \inf_{x \in \Omega} r(x,t) - \frac{M'(t)}{M(t)} \,\,\text{ changes sign}.
$$
}

\end{itemize}

%\item[{\rm(c)}] Fix two positive periodic functions $a(t)$ and $b(t)$, and consider the family 
%$$
%\Gamma_{a,b}:= \left\{(r,K):\, \frac{\overline{K}(t)}{\overline{(K/r)}(t)}=a(t)\, \text{ and }\, \frac{1}{\overline{(K/r)}(t)} = b(t)  \right\}, 
%$$
%Then since $M(t)$ depends only on $a(t)$ and $b(t)$ but not on other details of $K(x,t)$, \eqref{eq:R} becomes a restriction on the spatial heterogeneity of $K$ and $K/r$.
%%\begin{equation}\label{eq:RR}
%%K(x,t) - \overline{K}(t) > -M(t).
%%\end{equation}
%
%
%If we assume in addition that $\overline{(K/r)}(t) \equiv 1$ but $\overline{K}(t)$ is nonconstant, (i.e. $a(t) = \overline{K}(t)$ and $b(t) \equiv 1$), then \eqref{eq:MMM} becomes
%$$
%\frac{M'(t)}{M(t)} + M(t) = \overline{K}(t)
%$$
%so that $\overline{K}(t) - M(t)$ changes sign, and there exists {\color{red}$(r,K) \in \Gamma_{a,b}$ such that \eqref{eq:R} is not satisfied.(add details)} In such case, we will see in Subsection \ref{subsect:2.2} that IFD is impossible.

\end{remark}

\subsection{Necessarily condition for an IFD}\label{subsect:2.2}
{In the previous subsection, we saw that $\theta^*(x,t) = M(t)\tilde{K}(x,t)$ is an IFD. In this subsection we will show that this is the only possibility. For this purpose, consider the $N$-species competition model
\begin{equation}\label{eqn}
\begin{cases}
\frac{d u_i}{dt} = \nabla \cdot [ \mu_i(x,t) \nabla u_i  - u_i \vec{P}_i(x,t)] + r(x,t) u_i \left( 1- \frac{\sum_{j=1}^N u_j}{K(x,t)}\right) &\text{ in }\Omega\times(0,\infty), \text{ for }1 \leq i \leq N,\\
n \cdot [\mu_i(x,t)\nabla u_i - u_i \vec{P}_i(x,t)]=0 &\text{ on }\partial\Omega \times(0,\infty), \text{ for }1 \leq i \leq N,\\
u_i(x,0) = u_{i,0}(x) &\text{ in }\Omega,\text{ for }1 \leq i \leq N.
\end{cases}
\end{equation}
where $\mu_i(x,t), \vec{P}_i$ are $T$-periodic in $t$ and $r(x,t),K(x,t)$ are as before. 
\begin{definition}
Let $(\tilde{u}_i(x,t))_{i=1}^N$ be a positive, $T$-periodic solution of \eqref{eqn}. We say that $(\tilde{u}_i(x,t))_{i=1}^N$ is an IFD if $\theta^*(x,t)=\sum_{i=1}^N \tilde{u}_i(x,t)$ is an IFD, i.e.
$$
r(x,t) \left( 1 - \frac{\sum_{i=1}^N \tilde{u}_i(x,t)}{K(x,t)}\right) \quad \text{ depends on }t\text{ only}.
$$
\end{definition}
\begin{theorem}\label{prop:nec}
Let $(\tilde{u}_i(x,t))_{i=1}^N$ be a positive, $T$-periodic solution of \eqref{eqn}.
Suppose $(\tilde{u}_i(x,t))_{i=1}^N$ is an IFD, 
then \eqref{eq:R} holds and $$\sum_{i=1}^N\tilde{u}_i(x,t) = M(t)\tilde{K}(x,t),$$ where $M(t)$ and $\tilde{K}(x,t)$ are uniquely determined by $r(x,t)$ and $K(x,t)$, via Lemma \ref{lem:M} and \eqref{eq:m} respectively. 
\end{theorem}
\begin{remark}
As noted in Remark \ref{rmk:2.4}(d), there exists positive periodic $r(x,t),K(x,t)$ such that \eqref{eq:R} does not hold, and thus Corollary \ref{lem:necMm} shows that IFD is impossible for such environments.
\end{remark}
}
\begin{proof}[Proof of Theorem \ref{prop:nec}]
First, we define $\theta^*(x,t)=\sum_{i=1}^N \tilde{u}_i(x,t)$ and then define $M(t)$ as
\begin{equation}\label{eq:M0}
M(t) %= \bar{\theta^*}(0)\exp\left[ \int_0^t  F(s)\,ds\right] 
=  \overline{\theta^*}(0)\exp\left[ \int_0^t  r(x,s) \left( 1 - \frac{\theta^*(x,s)}{K(x,s)}\right)\,ds\right], 
\end{equation}
where $\overline{\theta^*}(0) = \frac{1}{|\Omega|}\int_\Omega \theta^*(x,0)\,dx$. (Note that the right hand side of \eqref{eq:M0} is independent of $x$, since $\theta^*$ is an IFD.) Then $M(t)$ is positive, and satisfies 
\begin{equation}\label{eq:thetaMM}
 \frac{M'(t)}{M(t)} = r(x,t) \left( 1 - \frac{\theta^*(x,t)}{K(x,t)}\right)
\end{equation}
and the equation of $\tilde{u}_i(x,t)$ becomes
\begin{equation}\label{eq:thetaM}
\frac{\partial \tilde{u}_i}{\partial t} = \nabla \cdot [\mu_i \tilde{u}_i - \tilde{u}_i \vec{P}_i] + \frac{M'(t)}{M(t)}\tilde{u}_i.
\end{equation}
Using the no-flux boundary conditions of $\tilde{u}_i$, we can integrate \eqref{eq:thetaM} in $\Omega$ to get
%Integrate \eqref{eq:thetaM} and using the boundary conditions of $\tilde{u}_i$, we have
\begin{equation*}%\label{eq:thetaMMM}
\frac{d}{dt} \overline{\tilde{u}_i}(t) = \frac{M'(t)}{M(t)} \overline{\tilde{u}_i}(t).
\end{equation*}
Adding in $i=1,...,N$, we obtain
\begin{equation}\label{eq:thetaMMM}
\frac{d}{dt} \overline{\theta}^*(t) = \frac{M'(t)}{M(t)} \overline{\theta}^*(t).
\end{equation}
Dividing \eqref{eq:thetaMMM} by $\overline{\theta^*}(t)>0$, and setting $t=0$ in \eqref{eq:M0}, we have
$$
(\log M)'(t) = (\log \overline{\theta^*})'(t)  \quad \text{ for all }t\geq 0, \quad \text{ and }\quad M(0) = \overline{\theta^*}(0).
$$
Therefore, 
\begin{equation}\label{eq:thetaMMMM}
\overline{\theta^*}(t) = M(t) \quad \text{ for all }t.
\end{equation}
In particular, $M(t)$ is $T$-periodic in $t$ as well.

Next, we show that $M(t)$ satisfies \eqref{eq:MMM}. Indeed,
multiply both sides of \eqref{eq:thetaMM} by $K(x,t)/r(x,t)$ to get
$$
\frac{K(x,t)}{r(x,t)} \frac{M'(t)}{M(t)} = K(x,t) - \theta^*(x,t).
$$
Integrate the above over in $x\in \Omega$, and use \eqref{eq:thetaMMMM}, we get
$$
\overline{K/r}(t) \frac{M'(t)}{M(t)} = \overline{K}(t) - M(t),
$$
which is equivalent to \eqref{eq:MMM}. 

Finally, define $\tilde{K}(x,t) = \theta^*(x,t)/M(t)$, we deduce from \eqref{eq:thetaMM} that
\begin{equation}\label{eq:R12}
\frac{M'(t)}{M(t)} = r(x,t) \left( 1 - \frac{M(t)\tilde{K}(x,t)}{K(x,t)}\right)
\end{equation}
which is equivalent to \eqref{eq:m}. %This proves the first part of the lemma.

Finally, \eqref{eq:R12} implies $r(x,t) > M'(t)/M(t)$, which implies that
\eqref{eq:R} holds. 
\end{proof}

%single species equation \eqref{eq:theta}.
%, by setting the second species to zero in the competition system \eqref{eq:system}.
\begin{corollary}\label{lem:necMm}
Let $\theta^*(x,t)$ be a positive solution of of the single species problem \eqref{eq:theta}. 
Suppose $\theta^*$ is an IFD, i.e.
$$
r(x,t) \left( 1 - \frac{\theta^*(x,t)}{K(x,t)}\right) \quad \text{ depends on }t\text{ only},
$$
then $\theta^*(x,t) = M(t)\tilde{K}(x,t)$ and \eqref{eq:R} holds, where $M(t)$ and $\tilde{K}(x,t)$ are uniquely determined by $r(x,t)$ and $K(x,t)$, via Lemma \ref{lem:M} and \eqref{eq:m} respectively. 
\end{corollary}

\subsection{Sufficient condition for existence of an IFD strategy}\label{subsect:2.3}
%\begin{description}
%%%\item[{\rm(R1)}] $\int_0^T a(t)\,dt >0$ and $b(t)>0$ for $t \in [0,T]$;
%
%\item[\eqref{eq:R}] for $M(t)$ given by \eqref{eq:MT}, 
%$$
%\frac{K(x,t) - \overline{K}(t)}{M(t)} >    \frac{\overline{K}(t)}{M(t) \overline{(K/r)}(t)} \left[ (K/r)(x,t) -  \overline{(K/r)}(t)\right] - \frac{(K/r)(x,t)}{\overline{(K/r)}(t)}
%$$
%for $x\in\Omega$ and $t \in [0,T]$.
%\end{description}

\begin{theorem}\label{thm:existence}
Let $r(x,t),K(x,t)$ be given, and let $M(t), \tilde{K}(x,t)$ be given respectively by \eqref{eq:MT} and \eqref{eq:m} in terms of $r(x,t)$ and $K(x,t)$. Suppose \eqref{eq:R} holds. Then for each $\mu(x,t)>0$, there exists a suitable $\vec{P}(x,t)$ such that the corresponding positive periodic solution $\theta^*(x,t)$ is an IFD. In fact, $\theta^*(x,t) = M(t)\tilde{K}(x,t)$.
\end{theorem}

\begin{proof}[Proof of Theorem \ref{thm:existence}]
Given  $r(x,t), K(x,t)$ and $\mu(x,t)$, and let $M(t)$ and $\tilde{K}(x,t)$ be given respectively by \eqref{eq:MT} and \eqref{eq:m}. Since \eqref{eq:R} holds, Lemma \ref{lem:IFD} says that the function $\theta^*(x,t)$ defines an IFD. 

Next, we define $\vec{P}^*(x,t)$ in terms of $\mu(x,t)$, as in \cite[Eqn. (8)]{Cantrell2018}:
\begin{equation}\label{eq:P}
\vec{P}^*(x,t) = \frac{1}{\tilde{K}(x,t)} \left[ \mu \nabla \tilde{K}(x,t) - \nabla q \right], 
\end{equation}
where, for each $t$, $q(x,t)$ is the unique solution to
$$
\Delta q = \frac{\partial \tilde{K}}{\partial t} \quad \text{ in } \Omega, \quad \text{ and }\quad n \cdot \nabla q = 0 \quad \text{ on }\partial\Omega.
$$
The function $q$ is well-defined for each $t$, since we have $\int_\Omega \frac{\partial \tilde{K}}{\partial t}(x,t)\,dx =0$ for all and $t$, by \eqref{eq:m2m}. 
\begin{claim}
The unique time-periodic solution corresponding to the dispersal strategy $(\mu, \vec{P}^*)$ is given by $\theta^*(x,t) = M(t)\tilde{K}(x,t)$. i.e. it is the unique positive solution to
\begin{empheq}[left = \empheqlbrace]{align}
&\frac{\partial \theta^*}{\partial t} = \nabla \cdot[ \mu \nabla \theta^* - \theta^*\vec{P}^*] + r(x,t)\theta^*\left[ 1 - \frac{\theta^*(x,t)}{K(x,t)}\right] &\text{ in }&\Omega\times[0,T], \label{eq:IFD1a}\\
&n \cdot [\mu \nabla \theta^* - \theta^* \vec{P}^*] = 0&\text{ on }&\partial\Omega \times [0,T], \label{eq:IFD1b}\\
&\theta^*(x,0) = \theta^*(x,T) & \text{ in }&\Omega. \label{eq:IFD1c}
\end{empheq}
%\begin{align}
%&\frac{\partial \theta^*}{\partial t} = \nabla \cdot[ \mu \nabla \theta^* - \theta^*\vec{P}^*] + r(x,t)\theta^*\left[ 1 - \frac{\theta^*(x,t)}{K(x,t)}\right] &\text{ in }&\Omega\times[0,T], \label{eq:IFD1a}\\
%&n \cdot [\mu \nabla \theta^* - \theta^* \vec{P}^*] = 0&\text{ on }&\partial\Omega \times [0,T], \label{eq:IFD1b}\\
%&\theta^*(x,0) = \theta^*(x,T) & \text{ in }&\Omega. \label{eq:IFD1c}
%\end{align}  
\end{claim}
To show \eqref{eq:IFD1c}, we first observe that $M(t)$ is $T$-periodic, by Lemma \ref{lem:M}. Thus $\tilde{K}(x,t)$ is also $T$-periodic in $t$, by \eqref{eq:m}. Since $\theta^*(x,t) = M(t)\tilde{K}(x,t)$, we deduce \eqref{eq:IFD1c}.

Next, we show \eqref{eq:IFD1b}. 
\begin{align*}
n \cdot [\mu \nabla \theta^* - \theta^* \vec{P}^*] &= M(t) n\cdot[\mu \nabla \tilde{K}- \tilde{K} \vec{P}^*]  \\
&= M(t) \left[ \mu (n \cdot \nabla \tilde{K}) - \tilde{K} (n \cdot \vec{P}^*)  \right]\\
&= M(t) \left[ \mu (n \cdot \nabla \tilde{K}) - \tilde{K} \left(\frac{n}{\tilde{K}} \right)\cdot (\mu \nabla \tilde{K} -\nabla q)\right]\\
&= M(t) (n \cdot  \nabla q) = 0 
\end{align*}
where the third equality follows by taking the scalar product of both sides of \eqref{eq:P} with the outer normal vector $n$, and the last equality follows from the fact that $q$ satisfies a Neumann boundary condition on $\partial\Omega$. 

Before going further, we multiply both sides of \eqref{eq:P} by $\tilde{K}$ and take divergence of both sides to obtain
\begin{equation}\label{eq:qmt}
\nabla\cdot(\tilde{K} \vec{P}) = \nabla \cdot(\mu \nabla \tilde{K} - \nabla q) = \nabla \cdot(\mu \nabla \tilde{K}) - \frac{\partial \tilde{K}}{\partial t}
\end{equation}
Now we proceed to show \eqref{eq:IFD1a}.
\begin{align*}
&\quad \frac{\partial \theta^*}{\partial t}-  \nabla \cdot [ \mu \nabla \theta^* - \theta^* \vec{P}^*]\\
&= \frac{\partial }{\partial t}(M\tilde{K}) -  \nabla \cdot [\mu \nabla (M(t)\tilde{K}(x,t)) - (M(t)\tilde{K}(x,t)) \nabla \vec{P}^*]\\
&=\tilde{K}(x,t)M'(t)  + M(t) \left\{\frac{\partial \tilde{K}}{\partial t}(x,t) - \nabla \cdot [\mu \nabla \tilde{K}(x,t) - \tilde{K}(x,t)  \vec{P}^*] \right\}\\
&= \tilde{K}(x,t) M'(t)
\end{align*}
where we used  \eqref{eq:qmt} in the last equality. Taking \eqref{eq:misifd} into account, 
$$
\frac{\partial \theta^*}{\partial t}-  \nabla \cdot [ \mu \nabla \theta^* - \theta^*\nabla \vec{P}^*] = r \tilde{K} M\left(1 - \frac{\theta^*}{K}\right). 
$$
This proves \eqref{eq:IFD1a}.
\end{proof}

\section{An IFD is necessary and sufficient for  ESS and NIS}\label{sec:3}

We consider the following class of reaction-diffusion-advection models in the spatially heterogeneous and temporally periodic setting:
\begin{equation}\label{eq:system}
\left\{\begin{array}{ll}
\frac{\partial u}{\partial t} = \nabla\cdot[\mu(x,t) \nabla u - u \vec{P}(x,t)] + r(x,t)u\left(1 - \frac{u + v}{K(x,t)}\right) &\text{ in }\Omega \times (0,\infty), \\
\frac{\partial v}{\partial t} = \nabla\cdot[\nu(x,t) \nabla v- v \vec{Q}(x,t)] +  r(x,t)v\left(1 - \frac{u + v}{K(x,t)}\right)  &\text{ in }\Omega \times (0,\infty), \\
n \cdot [\mu(t,x) \nabla u + u \vec{P}(x,t)] = 0   & \text{ on }\partial\Omega \times (0,\infty)\\
n \cdot [\nu(t,x) \nabla v + v \vec{Q}(x,t)] = 0   & \text{ on }\partial\Omega \times (0,\infty), 
\end{array}\right.
\end{equation}
where the functions $\mu,\nu, r, K$ are positive, $T$-periodic in $t$, and the vector fields $\vec{P},\vec{Q}$ are $T$-periodic in $t$; 
%subject to the boundary condition
%\begin{align}
%B_1[u] = n \cdot [\mu(t,x) \nabla u + u \vec{P}(x,t)] = 0   & \text{ on }\partial\Omega \times (0,\infty)   \label{eq:2a}\\
%B_2[v] = n \cdot [\nu(t,x) \nabla v + v \vec{Q}(x,t)] = 0   & \text{ on }\partial\Omega \times (0,\infty)   \label{eq:2b}
%\end{align}
and $n$ denotes the outer unit normal vector on $\partial\Omega$. We assume that there is no net movement into or out of the domain, as is reflected by an application of divergence theorem:
\begin{equation}
\int_\Omega \nabla\cdot[\mu(x,t) \nabla u - u \vec{P}(x,t)]\,dx = \int_{\partial\Omega} n \cdot  [\mu(t,x) \nabla u - u \vec{P}(x,t)]\,dx = 0.
\end{equation}
A similar equality holds for $v$.  
We note that the competition system \eqref{eq:system} has one trivial periodic solution $(0,0)$, and two semitrivial periodic solutions $(\theta^*,0)$ and $(0,v^*)$, where $\theta^*$ is the unique positive periodic solution of \eqref{eq:theta}, and $v^*$ is the unique positive solution to 
\begin{equation}\label{eq:vstar}
\left\{\begin{array}{ll}
\frac{\partial v^*}{\partial t} = \nabla \cdot [\nu(x,t) \nabla v^* - v^* \vec{Q}(x,t)] + r(x,t) v^*\left( 1 - \frac{v^*}{K(x,t)}\right) &\text{ in }\Omega\times [0,T],\\
n \cdot[\nu(x,t) \nabla v^* - v^* \vec{Q}(x,t)] = 0 &\text{ on }\partial\Omega \times [0,T],\\
v^*(x,0) = v^*(x,T) &\text{ in }\Omega.
\end{array}\right.
\end{equation}

There are two main results in this section. The first one is to show that, for the competition model \eqref{eq:system}, the dispersal strategies which generate IFD always dominate the strategies which do not. That is, the dispersal strategies which generates IFD are both ESS and NIS.

\begin{theorem}\label{thm:essnis}
Let $(u,v)$ be a solution of the competition model \eqref{eq:system} with nonnegative, nontrivial initial data. Suppose  
\begin{description}
\item[{\bf(C)}]  $\theta^*$ is IFD and $v^*$ is not IFD,
%\item[{\bf(C2)}] 
\end{description}
where IFD is defined in the sense of Definition \ref{def:IFD}. Then the semitrivial periodic solution $(\theta^*,0)$ is globally asymptotically stable among all nontrivial nonnegative initial data, i.e.
$$
\|(u(\cdot,t), v(\cdot,t)) - (\theta^*(\cdot,0), 0)\|_{C(\bar\Omega)} \to 0 \quad \text{ as } \,\, t \to \infty.
$$
\end{theorem}
Theorem \ref{thm:essnis} is proved in Subsection \ref{sec:4.1}.

The second main result says that the condition $\theta^*$ being an IFD in Theorem \ref{thm:essnis} is necessary.
\begin{theorem}\label{thm:nnifd}
Let $(\theta^*(x,t),0)$ be a $T$-periodic solution of \eqref{eq:system}. If $\theta^*(x,t)$ is not IFD, then there exists $(\mu,\vec{Q})$ such that a second species with diffusion rate $\mu$ and advection rate $\vec{Q}$ can invade when rare, i.e. the $(\theta^*(x,t),0)$ is unstable in \eqref{eq:system}. 
\end{theorem}

\begin{proof}
{Recall that $F(x,t) = r(x,t)\left(1 - \frac{\theta^*(x,t)}{K(x,t)}\right)$. 
Suppose $\theta^*(x,t)$ is not IFD, we claim that there exists a smooth, $T$-periodic curve $\gamma: \mathbb{R} \to {\rm Int}\,\Omega$ such that
\begin{equation}\label{eq:nnnifd}
\int_0^T F(\gamma(t),t)\,dt >0,
\end{equation}
Indeed, if $\theta^*(x,t)$ is not IFD, then 
there exists an open set $I \subset [0,T]$ such that $x \mapsto F(x,t)$ is nonconstant for each $t \in I$. Hence
$$
\int_0^T \sup_{x \in \Omega} F(x,t)\,dt > \frac{1}{\Omega} \int_0^T \int_\Omega F(x,t)\,dxdt =0
$$
where the last equality follows by integrating the equation \eqref{eq:theta} over $\Omega \times [0,T]$, and integration by parts. Hence, there exists a measurable function $\gamma_0: [0,T]$ such that 
$$
\int_0^T F(\gamma_0(t), t)\,dt = \int_0^T \sup_{x \in \Omega} F(x,t)\,dt >0.
$$
Now, let $\gamma_n$ be a sequence of smooth, $T$-periodic functions such that $\gamma_n \to \gamma_0$ a.e. (for instance, we extend $\gamma_0(t)$ periodically in $t$ and molify), then by the bounded convergence theorem, we have 
$$
\lim_{n\to\infty} \int_0^T F(\gamma_n(t),t)\,dt = \int_0^T F(\gamma_0(t), t)\,dt >0.
$$
Thus \eqref{eq:nnnifd} is proved, if we take $\gamma = \gamma_n$ for some $n$ sufficiently large.}

Now, let $\mu$ be a positive constant and define 
\begin{equation}
\vec{Q}(x,t) = \alpha(\gamma(t)-x),
\end{equation}
where $\alpha$ is a constant.
It suffices to show that the principal eigenvalue $\lambda_1$ of the linear periodic-parabolic problem 
\begin{equation}\label{eq:pevp1}
\begin{cases}
\frac{d\varphi}{dt} =  \nabla\cdot [\mu \nabla \varphi - \alpha\varphi (\gamma(t)-x)] + F(x,t) \varphi + \lambda_1 \varphi &\text{ in }\Omega \times [0,T],\\
n \cdot  [\nabla \varphi - \alpha\varphi (\gamma(t)-x)]=0 &\text{ on }\partial\Omega \times [0,T],\\
\varphi(x,0) = \varphi(x,T) &\text{ in }\Omega. 
\end{cases}
\end{equation}
is negative. 
We proceed by considering the adjoint problem, for which the principal eigenvalue $\lambda_1$ is the same.
\begin{equation}\label{eq:pevp1}
\begin{cases}
\frac{d\phi}{dt} =  \mu \Delta\phi + \alpha \nabla \phi \cdot(\tilde\gamma(t)-x)] + \tilde F(x,t) \phi + \lambda_1 \phi &\text{ in }\Omega \times [0,T],\\
n \cdot \nabla \phi =0 &\text{ on }\partial\Omega \times [0,T],\\
\phi(x,0) = \phi(x,T) &\text{ in }\Omega, 
\end{cases}
\end{equation}
where, after reversing time,  $\tilde\gamma(t) = \gamma(T-t)$ and $\tilde{F}(x,t) = F(x,T-t)$. By taking $m(x,t) = -\frac{1}{2}|x - \tilde\gamma(t)|^2$, {Remark \ref{rmk:a3} in the
Appendix shows the hypotheses of Proposition \ref{prop:a} in the Appendix are met
and thereby that result may be employed to show that}
$$
\limsup_{\alpha \to \infty} \lambda_1 \leq - \frac{1}{T}\int_0^T \tilde{F}(\tilde\gamma(t),t)\,dt = -\frac{1}{T}\int_0^T {F}(\gamma(t),t)\,dt. 
$$
Since the last term on the right hand side is negative (by \eqref{eq:nnnifd}), we have proved that the new species can indeed invade the non-ideal free distribution $\theta^*(x,t)$ when rare. 
\end{proof}

\begin{remark}
Let $N \geq 2$ and let $\tilde{u}=(\tilde{u}_1,...,\tilde{u}_{N})$ be a stable $T$-periodic solution of the $N$-species competition system \eqref{eqn}. Then the above can be applied to show the instablity of $(\tilde{u},0)$ in the extended $(N+1)$-species competition model, with an additional species with a suitable dispersal strategy.
\end{remark}
\begin{remark}
For environments where \eqref{eq:R} does not hold, it is proved in Theorem \ref{prop:nec} that no $T$-periodic solutions can achieve IFD. In such case, the proof of Theorem \ref{thm:nnifd} says that
no $T$-periodic solution consisting of any number of species can be an ESS. 
\end{remark}
\begin{remark}
{When the number of species is less than or equal to two, then the corresponding model generates a monotone dynamical system \cite{Smith1995}, in which any attractor is necessarily $T$-periodic. Hence, Theorem \ref{thm:nnifd} says that, in case \eqref{eq:R} does not hold, any ecological attractor that is ESS must not be periodic in time and must consist of at least three species.}
\end{remark}

\subsection{Proof of Theorem \ref{thm:essnis}}\label{sec:4.1}
In this subsection, we prove Theorem \ref{thm:essnis} by transforming the question into the context of \cite{Cantrell2018}, and then adapting the arguments therein. 
\begin{proof}[Proof of Theorem \ref{thm:essnis}]
By Corollary \ref{lem:necMm}, we deduce that \eqref{eq:R} holds and 
\begin{equation}\label{eq:C'}
\theta^*(x,t) = M(t)\tilde{K}(x,t) \quad \text{ and }\quad v^*(x,t) \not \equiv  M(t)\tilde{K}(x,t).
\end{equation}
The idea is to transform the system into one that is similar to \cite{Cantrell2018}, and adapt the arguments therein. For this purpose, write
$$
(u(x,t),v(x,t)) = M(t)(U(x,t), V(x,t))
$$
so that by \eqref{eq:m}, $(U,V)$ satisfies %the following problem
\begin{equation}\label{eq:UV}
\left\{
\begin{array}{ll}
\frac{\partial U}{\partial t} = \nabla\cdot[\mu(x,t) \nabla U - U \vec{P}(x,t)] + \hat r(x,t)U\left(\tilde{K}(x,t) - U - V\right) &\text{ in }\Omega \times (0,\infty)\\
\frac{\partial V}{\partial t} = \nabla\cdot[\nu(x,t) \nabla V - V \vec{Q}(x,t)] +  \hat{r}(x,t)V\left(\tilde{K}(x,t) - U-V\right)  &\text{ in }\Omega \times (0,\infty),\\
n \cdot [\mu(t,x) \nabla U - U \vec{P}(x,t)] = 0   & \text{ on }\partial\Omega \times (0,\infty),\\
n \cdot [\nu(t,x) \nabla V - V \vec{Q}(x,t)] = 0   & \text{ on }\partial\Omega \times (0,\infty)
\end{array}\right.
\end{equation}
where $\hat{r}(x,t) = \frac{r(x,t)M(t)}{K(x,t)}$ and $\tilde{K}(x,t)$ is given by \eqref{eq:m}.  
%Note that $\hat{r}>0$ and $\tilde{K}>0$, as the condition \eqref{eq:R} holds. 
Since $(\theta^*(x,t),0) = (M(t)\tilde{K}(x,t),0)$ and $(0,v^*(x,t))$ are semitrivial $T$-periodic solutions of \eqref{eq:system}, we deduce that
$(\tilde{K}(x,t),0)$ and $(0, V^*(x,t)) = (0, v^*(x,t)/M(t))$ are the corresponding semitrivial $T$-periodic solutions of the transformed system \eqref{eq:UV}.

Substituting $(\tilde{K}(x,t),0)$ into the first equation of \eqref{eq:UV}, we obtain
\begin{equation}\label{eq:m0}
\frac{\partial \tilde{K}}{\partial t} = \nabla \cdot [\mu \nabla \tilde{K} - \tilde{K} \vec{P}] \quad \text{ in }\Omega\times[0,T].
\end{equation}
It remains to show that $(\tilde{K}(x,t),0)$ is a globally asymptotically stable solution of \eqref{eq:UV}. 

\noindent {\bf Step 1.} We claim that \eqref{eq:UV} has no positive periodic solutions. To this end, let $({U},{V})$ be a nonnegative periodic solution of \eqref{eq:UV}, such that $U>0$. It remains to show that $V(x,t) = 0$.

Multiply the first equation of \eqref{eq:UV} by $\frac{\tilde{K}(x,t)}{U(x,t)}$, then multiply \eqref{eq:m0} by $\log U$, and add the two to get
\begin{align*}
\frac{\partial}{\partial t} (\tilde{K} \log U) &= \frac{\partial \tilde{K}}{\partial t} \log U + \frac{\tilde{K}}{U}\frac{\partial U}{\partial t}\\
&= \frac{\tilde{K}}{U}\nabla \cdot [\mu \nabla U - U\vec{P}] + \log U \nabla \cdot [\mu \nabla \tilde{K} -\tilde{K} \vec{P}] + \hat r \tilde{K}(\tilde{K} - U - V).
\end{align*}
Integrate by parts over $\Omega$, and then integrating in $[0,T]$, we get
\begin{align}
0 &= \iint \left\{ - \left[ \frac{\nabla \tilde{K}}{U} - \frac{\tilde{K}\nabla U}{U^2} \right]\left[\mu \nabla U - U\vec{P}\right] - \frac{\nabla U}{U}\left[ \mu \nabla \tilde{K} - \tilde{K} \vec{P}\right] + \hat{r}\tilde{K}(\tilde{K} - U - V)\right\}dx\,dt \notag\\
&= \iint \left\{ \mu \tilde{K} \frac{|\nabla U|^2}{U^2} - 2\mu \frac{\nabla \tilde{K}\cdot \nabla U}{U} + \nabla \tilde{K} \cdot \vec{P} + \hat{r}\tilde{K}(\tilde{K}  - U - V)\right\}dx\,dt \notag\\
&= \iint \left\{ \mu \tilde{K} \left| \frac{\nabla U}{U} - \frac{\nabla \tilde{K}}{\tilde{K}}\right|^2 - \mu \frac{|\nabla \tilde{K}|^2}{\tilde{K}} + \nabla \tilde{K} \cdot \vec{P}   + \hat{r} \tilde{K}(\tilde{K}  - U - V)\right\}dx\,dt \label{eq:sub1}
\end{align}
Next, we claim the following identity concerning $\tilde{K}(x,t)$ and $P(x,t)$  (see also \cite[P.10]{Cantrell2018}) 
\begin{equation}\label{eq:idcc}
\iint \left[ - \mu \frac{|\nabla \tilde{K}|^2}{\tilde{K}} + \nabla \tilde{K} \cdot \vec{P} \right]\,dxdt =0.
\end{equation}
Indeed, %multiply $\log m$ to both sides of \eqref{eq:m0}, and integrate by parts, we have
\begin{align*}
\int_\Omega \left[ - \mu \frac{|\nabla \tilde{K}|^2}{\tilde{K}} + \nabla \tilde{K} \cdot \vec{P}\right]\,dx &= -\int_\Omega (\nabla \log \tilde{K})\cdot [\mu \nabla \tilde{K} - \tilde{K} \vec{P}]\,dx \\
&= \int_\Omega \frac{\partial \tilde{K}}{\partial t} \log \tilde{K} \,dx \\
&=  \int_\Omega \frac{\partial}{\partial t} ( \tilde{K} \log \tilde{K} - \tilde{K})\,dx,
\end{align*}
where the first and third equalities follow by simply rewriting the integrands; the second equality follows by multiplying both sides of \eqref{eq:m0} by $\log \tilde{K}$ and integrating by parts. %Finally, the last expression is zero, as we are intergaring the derivative of a $T$-periodic.  This proves \eqref{eq:idcc}.
Integrating in $t \in [0,T]$, we obtain \eqref{eq:idcc}. 

Subsituting \eqref{eq:idcc} into \eqref{eq:sub1}, we deduce that
%$$
%0=\iint  \left\{ \mu m \left|\frac{\nabla U}{U} - \frac{\nabla m}{m}\right|^2+  \frac{\partial m}{\partial t} \log m +\hat{r}m( m- U - V)\right\}dx\,dt.
%%\\+ m_t \log m  +m\left[ r - U - V\right] \right\}dxdt
%$$
%Also, using identity $ \iint m_t \log m\,dtdx = \iint (m \log m - m)_t\,dtdx = 0$, which follows from periodicity of $m$, we get
\begin{equation}\label{eq:vdot1}
0=\iint  \left\{ \mu \tilde{K} \left| \frac{\nabla U}{U} - \frac{\nabla \tilde{K}}{\tilde{K}}\right|^2   +\hat r \tilde{K}( \tilde{K} - U - V)\right\}dx\,dt.
\end{equation}
Furthermore, we may integrate the first and second equations of \eqref{eq:UV} over $\Omega \times [0,T]$, then
\begin{equation}\label{eq:vdot2}
\iint \hat{r} U(\tilde{K}-U-V) \,dxdt = 0 = \iint  \hat{r} V(\tilde{K}-U-V)\,dxdt.
\end{equation}
Combining \eqref{eq:vdot1} and \eqref{eq:vdot2}, we have
\begin{equation}\label{eq:vdot}
0 = \iint  \left\{ \mu \tilde{K} \left| \frac{\nabla U}{U} - \frac{\nabla \tilde{K}}{\tilde{K}}\right|^2 + \hat{r}(\tilde{K} - U-V)^2\right\} \,dx\,dt.
\end{equation}
Therefore, 
\begin{equation}\label{eq:mmm0}
U(x,t)+ V(x,t) \equiv \tilde{K}(x,t)
\end{equation} 
and there exists a $T$-periodic function $0\leq c(t) \leq 1$ such that
\begin{equation}\label{eq:mmmm0}
U(x,t) = c(t)\tilde{K}(x,t) \quad \text{ in }\Omega \times [0,T].
\end{equation}
Substituting \eqref{eq:mmm0} and \eqref{eq:mmmm0} into the first equation of \eqref{eq:UV} and compare with \eqref{eq:m0}, we deduce that $c(t) = c_0$ for some constant $c_0 \in [0,1]$. Since both $U$ and $V$ are positive, we have $c_0 \in (0,1)$. Hence, we must have $V(x,t) = (1-c_0) \tilde{K}(x,t)$ for some $c_0 \in (0,1)$. 
Subsituting that into the second equation of \eqref{eq:UV}, we deduce that $\tilde{K}(x,t)$ is a positive, $T$-periodic solution of 
$$
\left\{
\begin{array}{ll}
\frac{\partial m}{\partial t} = \nabla \cdot \left[ \nu \nabla \tilde{K} - \tilde{K} \vec{Q}\right] = 0 \quad&\text{ in } \Omega\times [0,T],\\
n \cdot [\nu \nabla \tilde{K} - \tilde{K} \vec{Q}]  =0 &\text{ on }\partial\Omega \times [0,T].
\end{array}\right.
$$ 
This implies, by way of \eqref{eq:misifd}, %Lemma \ref{lem:IFD}, %\eqref{eq:m}, 
\begin{align*}
\frac{\partial}{\partial t}(M\tilde{K}) - \nabla \cdot \left[ \nu \nabla (M\tilde{K}) - (M\tilde{K}) \vec{Q}\right] &= M'(t)\tilde{K}(x,t)   \\
&= r(x,t)M(t)\tilde{K}(x,t) \left( 1- \frac{M(t)\tilde{K}(x,t)}{K}\right). 
%\text{ for }x \in \Omega,\, t\in [0,T],
\end{align*}
and that $n \cdot  [\nu \nabla (M\tilde{K}) - (M\tilde{K}) \vec{Q}]  =0$ on $\partial\Omega \times [0,T]$. 
By uniqueness of positive solution to \eqref{eq:vstar}, this means $v^*(x,t)=M(t)\tilde{K}(x,t)$. This is a contradiction to \eqref{eq:C'}. Hence we conclude that there exists no positive periodic solution $(U,V)$ to \eqref{eq:UV}. %This completes Step 1.

%
%Next, we denote $(0,v^*)$ to be the non-negative solution where species $u$ is excluded by species $v$, where $v^*(x,t)$ is the unique positive periodic solution of
%\begin{equation}\label{eq:vstar}
%\begin{cases}
%v_t = \left[ \nu v_x - vQ\right]_x  + rv\left(r-\frac{v}{K}\right) &\text{ for }x \in \Omega,\, t \in [0,T],\\
%\nu v_x - vQ=0 &\text{ for }x \in \partial\Omega,\, t\in[0,T],\\
%v(x,0) = v(x,T) &\text{ for }x \in \Omega.
%\end{cases}
%\end{equation}

\noindent {\bf Step 2.} The $T$-periodic solution $(0,V^*)$ of \eqref{eq:UV} is repelling. i.e. for solutions $(U,V)$ with non-negative, non-trivial initial data, $\|(U,V)(\cdot,t) - (0, V^*)(\cdot,t)\| \not\to 0$.

It suffices to show that $(0,V^*)$ is linearly unstable. (See, e.g. \cite[Theorem 1.3]{Lam2016}.) To this end, let $\sigma_1$ be the principal eigenvalue of the following periodic-parabolic eigenvalue problem (see, e.g. \cite{Hess1991} for the spectral theory of linear periodic-parabolic operators):
%\begin{equation}\label{eq:psi}
%\begin{cases}
%\psi_t = \left[ \mu \psi_x - \psi P\right]_x  + r\left(1-\frac{v^*}{K}\right)\psi + \sigma_1 \psi &\text{ for }x \in \Omega,\, t \in [0,T],\\
%\mu \psi_x - \psi P=0 &\text{ for }x \in \partial\Omega,\, t\in[0,T],\\
%\psi(x,0) = \psi(x,T) &\text{ for }x \in \Omega.
%\end{cases}
%\end{equation}
%Let $\psi(x,t) = M(t)\Psi(x,t)$, then the above problem can be transformed to 
\begin{equation}\label{eq:Psi}
\begin{cases}
\Psi_t = \nabla \cdot \left[ \mu \nabla \Psi - \Psi \vec{P}\right]  + \hat{r}(x,t)\left[\tilde{K}(x,t)- V^*(x,t)\right]\Psi + \sigma_1 \Psi &\text{ for }x \in \Omega,\, t \in [0,T],\\
n \cdot [\mu \nabla \Psi - \Psi\vec{P}]=0 &\text{ for }x \in \partial\Omega,\, t\in[0,T],\\
\Psi(x,0) = \Psi(x,T) &\text{ for }x \in \Omega,
\end{cases}
\end{equation}
where $V^*(x,t) = v^*(x,t)/M(t)$ is the unique positive periodic solution of
\begin{equation}\label{eq:Vstar}
\left\{
\begin{array}{ll}
V^*_t  = \left[ \nu V^*_x - V^* Q\right]_x + \hat rV^*( \tilde{K} - V^*) & \text{ in }\Omega \times [0,T],\\
n \cdot [\nu \nabla V^* - V^* \vec{Q}]  =0 &\text{ on }\partial\Omega \times [0,T].
\end{array}\right.
\end{equation}
%where $V^*(x,t) = \frac{v^*(x,t)}{M(t)}
By arguing as in Step 1, we obtain
$$
 = \iint \left\{ \mu m \left|\frac{\nabla \Psi}{\Psi} - \frac{\nabla \tilde{K}}{\tilde{K}}\right|^2 - \mu \frac{|\nabla \tilde{K}|^2}{\tilde{K}} + \nabla \tilde{K} \cdot \vec{P} + \hat{r} \tilde{K}(\tilde{K} - V^*) + \hat{r} \tilde{K} \sigma_1  \right\}dx\,dt.
$$
Using the identity \eqref{eq:idcc}, we dedeuce that
\begin{equation}\label{eq:id1}
0 = \iint \left\{\mu m \left|\frac{\nabla \Psi}{\Psi} - \frac{\nabla \tilde{K}}{\tilde{K}}\right|^2  + \hat{r} \tilde{K}(\tilde{K} - V^*) + \hat{r} \tilde{K} \sigma_1  \right\} dx\,dt.
\end{equation}
Furthermore, we may integrate \eqref{eq:Vstar} over $\Omega \times[0,T]$ to obtain
\begin{equation}\label{eq:id2}
\iint \hat{r}V^*(\tilde{K} - V^*)\,dx\,dt = 0.
\end{equation}
Subtracting \eqref{eq:id2} from \eqref{eq:id1}, we obtain
$$
0 = \iint \left\{\mu m \left|\frac{\nabla \Psi}{\Psi} - \frac{\nabla \tilde{K}}{\tilde{K}}\right|^2  + \hat{r} (\tilde{K} - V^*)^2 + \hat{r} \tilde{K} \sigma_1  \right\} dx\,dt.
$$
Hence $\sigma_1 \leq 0$. We claim that $\sigma_1 <0$. Indeed, suppose to the contrary that $\sigma_1 = 0$, then $V^* =\tilde{K}$ and thus $v^* = M(t)V^*(x,t) =M(t)\tilde{K}(x,t)$, which is impossible in view of \eqref{eq:C'}. Hence $\sigma_1 <0$, i.e. $(0,v^*)$ is linearly unstable.

\noindent {\bf Step 3.}  Conclude with theory of monotone dynamical systems. 

By \cite[Section 3]{Cantrell2018}, the system \eqref{eq:UV} defines a monotone dynamical system. 
By Step 1, we may invoke \cite[Theorem A(c)]{Hsu1996} (see also \cite[Theorem 1.3]{Lam2016}) to deduce that any internal trajectory of \eqref{eq:UV} {converges to either $(\tilde{K},0)$} or $(0,V^*)$. By Step 2, the possibility $(U,V) \to (0,V^*)$ is ruled out. Thus $(U,V) \to (\tilde{K},0)$. i.e., the $T$-periodic solution $(\tilde{K},0)$ of \eqref{eq:UV} is globally asymptotically stable. This is equivalent to the global asymptotic stability of the $T$-periodic solution $(\theta^*,0)$ of \eqref{eq:system}.
\end{proof}

\section{Conclusions}\label{sec:5}

      Mathematically, our results extend those of \cite{Averill2012, Cantrell2010, Cantrell2018} to more general forms of logistic-type models and to time-periodic environments where the total amount of resources available in the environment can vary in time.  Specifically, we extend the idea of fitness associated with a specific location in space to account for the full path that an organism takes in space and time over a periodic cycle, and extend the mathematical formulation of an ideal free distribution to general time periodic environments.  We find that as  in many  other cases,  populations using dispersal strategies that can produce  a (generalized) ideal free distribution have a competitive advantage relative to populations using strategies that do not produce an ideal free distribution.

   There are a number of biological conclusions that can be drawn from the results of our analysis.  First, 
a sufficient and necessary criterion (see \eqref{eq:R}) for the existence of such a generalized ideal free distribution is obtained. The criterion is a statement on the environment (local intrinsic growth rate $r(x,t)$ and local carrying capacity $K(x,t)$) only. It shows that while such generalized ideal free distribtuion is possible in quite general time-periodic environments, there exist certain environments where it is impossible for a single, or for that matter, any coalition of species, to achieve ideal free distribution. In the latter case, one can envision a never-ending dynamic succession of the community of species in the evolutionary timescale. Namely, starting with any group of interacting species which have already reached some stable pattern in the ecological timescale, there is always the possibility that a novel mutation arises (or a foreign invasive species arrives) and destabilizes the configuration. At the onset of such an instability, it may happen that the community becomes larger (one more species), or it may happen that one or more species becomes extinct. In either case, the whole community will then approach a different stable ecological configuration, until the next mutation/invasion takes place. 
    %(Chris and Steve: I don't know if this analogy is accurate ---- It is like the game of musical chair, but it never ends.)}

   In the environments where such a generalized ideal free distribution exists, we showed that there exist dispersal strategies that can be identified as producing a time-periodic version of an ideal free distribution, and such strategies are evolutionarily steady and are neighborhood invaders  from the viewpoint of adaptive dynamics.  Those strategies are thus optimal in a rather strong sense. However, implementing those strategies requires that organisms use information that is nonlocal in time and space and process that information in a sophisticated way.   In environments that are static in time but vary in space, organisms can attain an ideal free distribution on the basis of purely local information on the environment, {but nonlocal information is required in time periodic environments. This is already the case in
 \cite{Cantrell2018} when the total amount of resources available is constant in time. } There is evidence that organisms  can and do use nonlocal information in deciding how to move; see \cite{Fagan}.   However, using direct environmental information to disperse optimally  in time periodic environments where the total amount of resources in the overall environment varies in time seems to require information processing  that is equivalent to solving an ordinary differential equation.   That would appear to put it beyond the reach of most organisms.  However, there are less direct mechanisms including memory and social learning that might make optimal dispersal feasible. Also, it is worth noting that bats and some cetaceans are able to use echolocation to maneuver in real time, which in a sense amounts to solving scattering theory problems on the fly. Humans can only do that by {means of} advanced technology and/or mathematics.   There is empirical evidence that some whales can accurately track the places and times where resources will be most reliably abundant over periods of years, effectively averaging out yearly variations by using memory and resource tracking \cite{whales}.  The general question of how real organisms can disperse in something close to an ideal way remains as a major challenge in understanding the evolution and ecology of dispersal.  
  
An interesting observation is that a mathematical definition of fitness associated with a movement strategy for organisms in time periodic-environments requires consideration of the specific paths individuals take through space and time, not just the spatiotemporal distribution of the population that the strategy produces.   In a related but different context, it was shown in \cite {Lopez2019} that the effectiveness of a protection zone depends on the details of its path through space and time during a periodic cycle, where in the static case there are no such geometric restrictions in the static case.  
Also, in \cite{Liu2020b} the asymptotic behavior of the periodic-parabolic eigenvalue is also connected with some periodic cycle in the associated kineitic system. 
These are only a few data points, but they suggest a possible feature that may distinguish periodic-parabolic models from elliptic models. It may be that to extend results in the elliptic case that depend on  local conditions  in space to the periodic case will often require that the condition hold on or in a neighborhood of an entire  path connecting the beginning and  end of a periodic cycle.  That would be consistent with our conclusions about the need for nonlocal information in optimal dispersal in time periodic environments.  Related ideas are discussed in  \cite {Lopez2019}.
  
  There are many direction for further research on the topics studied  here.  On the mathematical side one obvious but possibly challenging direction would be to consider full consumer-resource models rather that just logistic models.  Another would be to consider nonlocal integrodifferential models and/or patch models in time periodic environments.  On the biological side, although we have obtained results on what sort of dispersal strategies are optimal, the biological interpretations and implications of our results are not all immediately clear, so  it would be of interest to explore those more deeply.

\appendix

\section{Asymptotic behavior of the principal eigenvalue}
Let $\lambda_1$ be the principal eigenvalue of 
\begin{equation}\label{eq:a1}
\begin{cases}
\frac{d\varphi}{dt} =  \mu\Delta\varphi + \alpha \nabla m(x,t) \cdot \nabla \varphi + V(x,t) \varphi + \lambda_1 \varphi &\text{ in }\Omega \times [0,T],\\
n \cdot \nabla \varphi=0 &\text{ on }\partial\Omega \times [0,T],\\
\varphi(x,0) = \varphi(x,T) &\text{ in }\Omega, 
\end{cases}
\end{equation}
where $\alpha$ is a positive constant, $m \in C^2(\overline\Omega\times[0,T])$ and $V \in C(\overline\Omega\times[0,T])$.
We will follow the idea in \cite[Propositions 2.1 and 2.2]{Liu2020a} to prove the following result.
\begin{proposition}\label{prop:a}
Suppose there exist a smooth, $T$-periodic curve $\gamma: \mathbb{R} \to {\rm Int}\,\Omega$ and a set $U$ which is open relative to $\overline\Omega \times [0,T]$ such that
\begin{itemize}
\item[{\rm(i)}] $\{(\gamma(t),t):\, t \in [0,T]\} \subset U \subset {\rm Int}\, \Omega \times [0,T]$; 
\item[{\rm(ii)}] $m(x,t) < m(\gamma(t),t)$ if $(x,t) \in U$ and $x \neq \gamma(t)$;
\item[{\rm(iii)}] $|\nabla m(x,t)|>0$ if $(x,t) \in U$ and $x \neq \gamma(t)$.
\end{itemize}
Then
%the mapping $x \mapsto m(x,t)$ attains a strict local maximum point at $x =\gamma(t)$, and that $\gamma(t)$ is an isolated critical point 
%then 
\begin{equation}\label{eq:a2}
\limsup_{\alpha \to \infty} \lambda_1 \leq - \frac{1}{T}\int_0^T V(\gamma(t),t))\,dt.
\end{equation}
\end{proposition}
\begin{remark}\label{rmk:a3}
Suppose $\gamma(t)$ is a smooth, $T$-periodic curve such that $\nabla^2 m(\gamma(t),t)$ is negative definite for each $t \in [0,T]$, then the hypothesis (i)-(iii) can be verified. % In fact, our proof can be moditied so that 
\end{remark}
\begin{proof}
%We will first prove the special case when, for each $t$, $m(x,t)$ attains its global spatial maximum at the unique point $x=\gamma(t)$. We will show how to relax this assumption at the end of the proof.
By a rescaling in the $x$ variable, we may assume $\mu=1$. 
We also assume, without loss of generality, that 
\begin{equation}\label{eq:a3}
m(\gamma(t),t)\equiv 0 \quad \text{ and }\quad \int_0^T V(\gamma(t),t)\,dt = 0.
\end{equation}
These can be achieved if we replace $m(x,t)$ by $m(x,t) - m(\gamma(t),t)$, $V(x,t)$ by $V(x,t) - \frac{1}{T}\int_0^T V(\gamma(t),t)\,dt$, and $\lambda_1$ by $\lambda_1 + \frac{1}{T}\int_0^T V(\gamma(t),t)\,dt$. 

The goal is to show that $\limsup\limits_{\alpha \to \infty} \lambda_1 <\ep$ for each $\ep>0$. By \cite[Proposition A.1]{Liu2020a}, it suffices to construct a non-negative, non-trivial subsolution $\underline\varphi$, such that 
\begin{empheq}[left = \empheqlbrace]{align}
&\frac{d\underline\varphi}{dt} \leq   \Delta\underline\varphi + \alpha \nabla m(x,t) \cdot \nabla \underline\varphi + V(x,t) \underline\varphi + \ep \underline\varphi &\text{ in }\Omega \times [0,T], \label{eq:a1.1}\\
&n \cdot \nabla \underline\varphi =0 &\text{ on }\partial\Omega \times [0,T], \label{eq:a1.2}\\
&\varphi(x,0) = \varphi(x,T) &\text{ in }\Omega. \label{eq:a1.3}
\end{empheq}
%\begin{align}
%&\frac{d\underline\varphi}{dt} \leq   \Delta\underline\varphi + \alpha \nabla m(x,t) \cdot \nabla \underline\varphi + V(x,t) \underline\varphi + \ep \underline\varphi &\text{ in }\Omega \times [0,T], \label{eq:a1.1}\\
%&\partial_n \underline\varphi =0 &\text{ on }\partial\Omega \times [0,T], \label{eq:a1.2}\\
%&\varphi(x,0) = \varphi(x,T) &\text{ in }\Omega. \label{eq:a1.3}
%\end{align}
By replacing $\underline\varphi(x,t)$ by $\underline\varphi(x,t)\exp(-\int_0^t V(\gamma(s),s)\,ds)$, we may strengthen \eqref{eq:a3} to 
\begin{equation}\label{eq:a4}
m(\gamma(t),t)\equiv 0 \quad \text{ and }\quad V(\gamma(t),t) \equiv  0.
\end{equation}
We now fix $\epsilon>0$ hereafter and construct $\underline\varphi$. By the strictness of the spatial local maximum point of $m$ at $x=\gamma(t)$, we can choose $a_0<0$ 
%and an open set $U$ containing $\{(x,t):\, x = \gamma(t)\}$, 
such that $m(x,t) \leq a_0$ on $\partial U$. By hypotheses (i) and (ii), 
%$$
%\bigcap_{a \in (a_0,0)} \{(x,t) \in U: m(x,t) > a\} = \{(\gamma(t),t):\, t \in [0,T]\}.
%$$
%Thus, 
we can choose $a $ small enough such that $a_0 < a < 0$ and
%By Sard's Theorem \cite{Sard1942}, we can choose $a,b$ small such that $a_0 <a<b<0$ and
\begin{equation}\label{eq:a4b}
|V(x,t)| < \frac{\ep}{2} \quad \text{ in }\quad \{(x,t) \in U:\, m(x,t) >a\}. %\label{eq:a4b}
\end{equation}
%and, for each $m_0 \in [a,b]$, $m_0$ is regular value of $m$\footnote{This is possible since the set of regular values of $m$ is open and the complement has zero Lebesgue measure.}, and that 
%the set $\{(x,t) \in  U: \,  m(x,t) >m_0 \}$ has compact smooth boundary in the open set $U$.
%By virtue of the fact that $[a,b]$ belongs to the set of regular values,
Next, fix $b=a/2$, then $a<b<0$ and, 
by virtue of hypothesis (iii) and compactness,  
\begin{equation}
\inf |\nabla m|>0 \quad \text{with the infimum taken over }\quad \{(x,t) \in U:\, a \leq m(x,t) \leq b\}. \label{eq:a4a}\\
\end{equation}
%and
%\begin{equation}\label{eq:a4b}
%&|V(x,t)| < \frac{\ep}{2} \quad \text{ in }\quad \{(x,t):\, m(x,t) >2a\}. \label{eq:a4b}
%\end{align}
Next, fix $\delta>0$ small such that
\begin{equation}\label{eq:adelta}
\delta( \|\partial_t m\|_{L^\infty(\Omega\times[0,T])} + \|\Delta m\|_{L^\infty(\Omega\times[0,T])}) < \frac{\ep}{2}
\end{equation}
and choose a smooth function $g :  \mathbb{R}\to \mathbb{R}$ such that 
\begin{equation}\label{eq:a4c}
g'(s)>0\,\,\text{ for }s  \in \mathbb{R}, \quad g(a) = 0, \quad  \text{ and }\quad g(s) = 1+\delta(s-b)\,\,\text{ for } s\in [b,0].
\end{equation}
Now, define
$$
\underline\varphi(x,t) = \begin{cases}
\max\{g(m(x,t)),0\} &\text{ in } U,\\
0 &\text{ otherwise.}
\end{cases}
$$

Since 
$m(x,t) \leq a_0$ on $\partial U$, it also satisfies $m(x,t) <a$ in a neighborhood of $\partial U$. Hence $\underline\varphi$ is continuous in $\overline\Omega \times [0,T]$, and satisfies \eqref{eq:a1.1} in an open set containing the complement of $U$.  Also, since $\underline\varphi(x,t)$ is $T$-periodic (since $m(x,t)$ is), and is compactly supported in the interior of $U$, 
t is clear that \eqref{eq:a1.2} and \eqref{eq:a1.3} hold. It remains to show \eqref{eq:a1.1} in $U$. 

For $(x,t) \in U$, the function $\underline\varphi$ can be written as %the maximum of $g(m(x,t))$ and $0$ in $U$, and that
$$
\underline\varphi(x,t) =\max\{g(m(x,t)),0\} =\begin{cases}
g(m(x,t)) &\text{ when  } m(x,t) >a,\\
0 &\text{ when }m(x,t) \leq a,
\end{cases}
$$ 
it is enough to verify $g(m(x,t))$ satisfies \eqref{eq:a1.1} 
in $\Omega_0:=\{(x,t) \in U:\, m(x,t)  >a\}$ in the classical sense. 
%Indeed, wherever $\underline\varphi$ is smooth, it satisfies \eqref{eq:a1.1} in the classical sense, and if $\underline\varphi$ is not differentiable at a point $(x,t)$, then it can locally be written as the maximum of $g(m(x,t))$ and $0$, both of which are classical subsolutions to \eqref{eq:a1.1}.
%Hence the differential inequality \eqref{eq:a1.1} is satisfied in $\Omega\times[0,T]$ in the weak sense, 
%provided we can verify that $\underline\varphi$ satisfies \eqref{eq:a1.1} in $\Omega_0:=\{(x,t) \in U:\, m(x,t) >a\}$. 
%
We argue differently for the two regions 
$$
\{(x,t) \in U:\, m(x,t) >b\} \quad \text{ and }\quad \{(x,t) \in U:\, a <m(x,t)  \leq b\}.
$$ 
In the first case $\underline\varphi(x,t) = 1+\delta (m(x,t)-b)$, and $\underline\varphi(x,t)\geq 1$, and that
\begin{align*}
&\quad \frac{d\underline\varphi}{dt} -   \Delta\underline\varphi - \alpha \nabla m(x,t) \cdot \nabla \underline\varphi - (V(x,t)+\ep) \underline\varphi \\
&=   \delta (\partial_t m -   \Delta m - \alpha |\nabla m|^2) - (V(x,t)+\ep) [1+\delta (m-b)]\\
&\leq \delta( |\partial_t m| + |\Delta m|) - \frac{\ep}{2} <0
\end{align*}
where we used \eqref{eq:a4b} for the second last inequality, and \eqref{eq:adelta} for the last inequality. In the latter case, we have $a <m(x,t) \leq b$, and that for $\alpha$ sufficient large,
\begin{align*}
&\quad \frac{d\underline\varphi}{dt} -   \Delta\underline\varphi - \alpha \nabla m(x,t) \cdot \nabla \underline\varphi - (V(x,t)+\ep) \underline\varphi \\
&= -g'(m)\left[-\partial_t m + \Delta m + \left(\frac{g''(m)}{g'(m)}+ \alpha \right)|\nabla m|^2\right]  - (V(x,t)+\ep) g(m)\\
&\leq -g'(m)\left[ - \|\partial_t m\|_\infty - \|\Delta m\|_\infty + \frac{\alpha}{2} (\inf |\nabla m|)^2\right]
\end{align*}
where we used \eqref{eq:a4b} and that $g(m)$, $g'(m)$ are both positive.
%for $\alpha$ sufficiently large (so that the expression in the squre bracket is positive). 
Note that the infimum of $|\nabla m|$, which is taken over the set $\{(x,t) \in U:\, a <m(x,t) \leq b\}$, is a positive real number by \eqref{eq:a4a}. Hence, 
%$$
%\inf g'(m) >0 \quad \text{ and }\quad \inf |\nabla m| >0.
%$$
by taking $\alpha$ sufficiently large, we deduce that \eqref{eq:a1.1} holds in $\{(x,t):\, a <m(x,t) \leq b\}$. Having verified that $\underline\varphi$ is a non-trivial, non-negative subsolution, we apply \cite[Proposition A.1]{Liu2020a} to deduce that 
$\limsup\limits_{\alpha \to \infty} \lambda_1 < \ep$. Since $\ep>0$ is arbitrarily, the proof is finished.
% for the special case when $m(x,t)$ attains its global spatial maximum at the unique point $x=\gamma(t)$ for each $t$.
%
%Finally, we consider the general case when $x=\gamma(t)$ is only a non-degenerate local maximum point.  
%In this case, we note that the above subsolution $\underline\varphi$ is supported only in a neighborhood of $\{(x,t):\, x=\gamma(t)\}$, so that the above construction is still valid. 
\end{proof}
\begin{remark}
The case $m(x,t) = h(x-\gamma(t))$, where $h:\mathbb{R}^n \to \mathbb{R}$ attains a strict local maximum point at $0$, can also be treated with a slight modification of the proof.  The key is to take $a<b<0$ to be small enough and such that $[a,b]$ belongs to the set of regular values of $h$.
\end{remark}

\bibliographystyle{siamplain}

\end{document}